\documentclass{amsart}

\usepackage{amsmath,amssymb}
\usepackage{latexsym}
\usepackage{graphicx}
\usepackage{color}
\usepackage{subfigure}

\newcommand{\NN}{\mathbb{N}}

\newcommand{\RR}{\mathbb{R}}

\newcommand{\MM}{\mathcal{M}}

\newcommand{\Ric}{\textup{Ric}}
\newcommand{\Vol}{\textup{Vol}}

\newcommand{\ud}{\textup{d}}

\newcommand{\HD}{\textup{HD}}
\newcommand{\Span}{{\tt{span}}}

\newtheorem{thm}{Theorem}[section]
\newtheorem{lem}[thm]{Lemma}

\title[Vector Diffusion Map]{Embedding Riemannian Manifolds by the Heat Kernel of the Connection Laplacian}
\author{
Hau-Tieng~Wu
}
\thanks{Department of Mathematics, University of Toronto, Toronto Ontario, Canada, email: hauwu@math.toronto.edu}
\date{}

\begin{document}
\maketitle

\begin{abstract}
Given a class of closed Riemannian manifolds with prescribed geometric conditions, we introduce an embedding of the manifolds into $\ell^2$ based on the heat kernel of the Connection Laplacian associated with the Levi-Civita connection on the tangent bundle. As a result, we can construct a distance in this class which leads to a pre-compactness theorem on the class under consideration.
\end{abstract}

\section{Introduction}

In \cite{berard}, the following class of closed Riemannian manifolds $\MM_{d,k,D}$ with prescribed geometric constrains are considered:
\begin{equation}
\MM_{d,k,D}=\{(M,g)| \dim(M)=d,\,\Ric(g)\geq (d-1)kg,\,\mbox{diam}(M)\leq D\},
\end{equation}
where $\Ric$ is the Ricci curvature and $\mbox{diam}$ is the diameter.
The authors embed $M\in \MM_{d,k,D}$ into the space $\ell^2$ of real-valued, square integrable series by considering the heat kernel of the Laplace-Beltrami operator of $M$. A distance on $\MM_{d,k,D}$, referred to as the {\em spectral distance} (SD), is then introduced based on the embedding so that the set of isometric classes in $\MM_{d,k,D}$ is precompact. Over the past decades many works in the manifold learning field benefit from this embedding scheme, for example, the eigenmaps \cite{belkin_niyogi:2003}, diffusion maps (DM) \cite{Coifman20065} and the manifold parameterizations \cite{jones2007,Jones_Maggioni_Schul:2010}. An important practical and theoretical question about how many eigenfunctions one needs to embed the manifold via DM has been answered in \cite{Bates:2014,Portegies:2013}. 

Recently, a new mathematical framework, referred to as {\em vector diffusion maps} (VDM), for organizing and analyzing massive high dimensional data sets, images and shapes was introduced in \cite{vdm}. In brief, VDM is a mathematical generalization of DM, which numerically is based on a generalization of graph Laplacian, referred to as {\em graph connection Laplacian} (GCL). 
While DM are based on the heat kernel of the Laplace-Beltrami operator over the manifold, VDM is based on the heat kernel of the connection Laplacian associated with the Levi-Civita connection on the tangent bundle of the manifold. The introduction of VDM was motivated by the problem of finding an efficient way to organize complex data sets, embed them in a low dimensional space, and interpolate and regress vector fields over the data. In particular, it equips the data with a metric, which we refer to as the {\em vector diffusion distance} (VDD). VDM has been applied to several different problems, ranging from the cryo-electron microscopy problem \cite{hadanisinger2011,vdmapp2011,Zhao_Singer:2014} to graph realization problems \cite{Cucuringu_Lipman_Singer:2012,Cucuringu_Singer_Cowburn:2012} and modern light source imaging technique \cite{Alexeev_Bandeira_Fickus_Mixon:2013,Marchesini_Tu_Wu:2014}. The application of VDM to the cryo-electron microscopy problem, which is aimed to reconstruct the three dimensional geometric structure of the macromolecule, provides a better organization of the given noisy projection images, and hence a better reconstruction result \cite{hadanisinger2011,vdmapp2011,Zhao_Singer:2014}. In addition, the GCL can be slightly modified to determine the orientability of a manifold and obtain its orientable double covering if the manifold is non-orientable \cite{singer_wu:2011}. 

In this paper, we consider the same class of closed Riemannian manifolds $\MM_{d,k,D}$ and focus on the connection Laplacian associated with the Levi-Civita connection on the tangent bundle. We analyze how the VDM embeds the manifold $M\in \MM_{d,k,D}$ into $\ell^2$ based on the heat kernel of the connection Laplacian of the tangent bundle. Based on the VDM and VDD, we introduce a new distance on $\MM_{d,k,D}$, referred to as {\em vector spectral distance} (VSD), which leads to the pre-compactness of the set $\MM_{d,k,D}$. 

\subsection{Summary of the work and organization of the paper} 
The results of this paper are organized in the following way. In Section \ref{section:background} the background material and notations are provided. 

In Section \ref{section:vdm}, the VDM and VDD are defined  and we discuss the embedding property of VDM in Theorem \ref{Vtembedding} and the local geodesic estimation property of VDD in Theorem \ref{thm:geod}. Indeed, VDM is a diffeomorphic embedding of a manifold $M\in\MM_{d,k,D}$ into the Hilbert space $\ell^2$, and if  $x,y\in M$ are close enough in the sense of geodesic distance, then the VDD of these two points are closely related to the geodesic distance. 

In Section \ref{section:vsd} we define VSD in the manifold set $\MM_{d,k,D}$. 
The key ingredients in this section are the generalized Kato's type inequality comparing the trace of the heat kernel of the Laplace-Beltrami operator and the trace of the heat kernel of the connection Laplacian and a nice isoperimetric inequality for heat kernel comparisons. In brief, although the behavior of the heat kernel of the connection Laplacian might be intricate, we may control its trace simply by the trace of the heat kernel of the Laplace-Beltrami operator. With these key ingredients, in Theorem \ref{mainthm} we show that the newly defined metric, the VSD, is really a distance between isometry classes of Riemannian manifolds in $\MM_{d,k,D}$. In Section \ref{section:precompact}, with the help of the VSD, the pre-compactness of the manifold set $\MM_{d,k,D}$ is derived in Theorem \ref{Theorem:Precompactness} from the Rellich's Theorem.  
To show this Theorem, we also need the following Lemma:
\begin{lem}\cite[Lemma 15]{berard}\label{precompactlemma}
Let $(E,\delta)$ be a metric space. Let $\mathcal{F}(E)$ denote the set of non-empty closed subsets of $E$, equipped with the Hausdorff distance
$h_\delta$ associated with $\delta$. If the metric space $(E,\delta)$ is precompact, so is the metric space
$(\mathcal{F}(E),h_\delta)$.
\end{lem}
In fact, we view the VDM of a given manifold in $\MM_{d,k,D}$ as a point in the set consisting of embeddings of all manifolds in $\MM_{d,k,D}$, and then apply Lemma \ref{precompactlemma} to show the pre-compactness of $\MM_{d,k,D}$ with related to the VSD. 
The main theorems are summarized here for the reader's convenience:
\begin{enumerate}
\item [$\bullet$] Theorem 4.6: For any fixed $t>0$, $d_t$, the VSD defined in (\ref{VtM}), is a distance between isometry classes of Riemannian manifolds in $\MM_{d,k,D}$. In particular, two Riemannian manifolds $(M,g),(M',g')\in \MM_{d,k,D}$ satisfy $d_{t}(M, M')=0$ if and only if $M$ and $M'$ are isometric.
\item [$\bullet$] Theorem 5.1: For any $t>0$, the space of the isometry classes in $\MM_{d,k,D}$ is $d_t$-precompact.
\end{enumerate}

In Section \ref{section:manifoldLearning}, we discuss the application of VDM and this new pre-compactness result from the viewpoint of manifold learning. 

\section{Background material}\label{section:background}
Let $(M,g)$ be a closed Riemannian manifold and $TM$ the tangent bundle. Denote $C^\infty(TM)$ the smooth vector fields and $L^2(TM)$ the vector fields satisfying
\begin{equation}
\int_M \langle X,X\rangle(x) \ud V(x)\leq \infty,
\end{equation}
where $\ud V$ is the volume form associated with $g$ and $\langle X,X\rangle(x):=g(X(x),X(x))$. Denote $\nabla$ the Levi-Civita connection of $\MM$ and $P_{x,y}$ the parallel transport from $y$ to $x$ via the geodesic linking them. Denote $\Delta^{\texttt{TM}}$ the connection Laplacian associated with $\nabla$ on the tangent bundle $TM$ \cite{gilkey}. The connection Laplacian $\Delta^{\texttt{TM}}$ is a self-adjoint, second order elliptic operator \cite{gilkey}. From the classical elliptic theory \cite{gilkey} we know that the heat semigroup, $e^{t\Delta^{\texttt{TM}}}$, $t>0$, with the infinitesimal generator $\Delta^{\texttt{TM}}$ is a family of self-adjoint operators with the heat kernel $k_{TM}(t,x,y)$ so that
\begin{equation}
e^{t\Delta^{\texttt{TM}}}X(x)=\int_ M k_{TM}(t,x,y)X(y)\ud V(y).
\end{equation}
The heat kernel $k_{TM}(t,x,y)$ is smooth in $x$ and $y$ and analytic in $t$ \cite{gilkey}.

It is well known \cite{gilkey} that the spectrum of $\Delta^{\texttt{TM}}$ is discrete inside $\RR^-$, the non-positive real numbers, and the only possible accumulation point is $-\infty$. We will denote the spectrum of $\Delta^{\texttt{TM}}$ as $\{-\lambda_k\}_{k=1}^\infty$, where $0=\lambda_0<\lambda_1\leq\lambda_2\ldots$, and its eigen-vector fields as $\{X_k\}_{k=1}^\infty$. Notice that $\lambda_0$ may not exist due to the topological obstruction. For example, we can not find a nowhere non-vanishing vector field on $S^2$. In other words, $\Delta^{\texttt{TM}} X_k=-\lambda_kX_k$ for all $k=1,2,\ldots$. It is also well known \cite{gilkey} that $\{X_k\}_{k=1}^\infty$ form an orthonormal basis for $L^2(T M)$. Denote the heat kernel of $\Delta^{\texttt{TM}}$ by $k_{TM}(t,x,y)$, which can be expressed as \cite{gilkey}:
\begin{equation}
k_{TM}(t,x,y)=\sum_{n=1}^\infty e^{-\lambda_n t}X_n(x)\otimes\overline{X_n(y)}.
\end{equation}

A calculation of the Hilbert-Schmidt norm of the heat kernel at $(t,x,y)$ gives
\begin{eqnarray}
\label{KF}
\|k_{TM}(t,x,y)\|_{HS}^2 &=& \operatorname{Tr}\left[k_{TM}(t,x,y)\overline{k_{TM}(t,x,y)} \right] \nonumber \\
&=& \sum_{n,m=1}^\infty e^{-(\lambda_n+\lambda_m)t} \langle X_n(x), X_m(x) \rangle \overline{\langle X_n(y), X_m(y) \rangle}.
\end{eqnarray}

On the other hand, the classical elliptic theory \cite{gilkey} allows us to decompose $L^2(T M)$ as $L^2(T M)=\oplus_{k=1}^\infty E_k$, where $E_k$ is the eigenspace of $\Delta^{\texttt{TM}}$ corresponding to the eigenvalue $-\lambda_k$. Denote by $m(-\lambda_k)$ the multiplicity of $-\lambda_k$. It is also well known that $m(-\lambda_k)$ is finite. Denote $\mathcal{B}(E_k)$ the set of bases of $E_k$, which is identical to the orthogonal group $O(m(-\lambda_k))$. Denote the set of the corresponding orthonormal bases of $L^2(T M)$ by
\begin{equation}
\mathcal{B}(M,g)=\Pi^\infty_{k=1}\mathcal{B}(E_k).
\end{equation}
By Tychonoff's theorem, we know $\mathcal{B}(M,g)$ is compact since $O(m(-\lambda_k))$ is compact for all $k\in \NN$. Also note that the dot products $\langle X_n(x), X_m(x) \rangle$, where $n,m\in\NN$, are invariant to the choice of basis for $T_x M$.

\section{Vector Diffusion Maps}\label{section:vdm}
Based on these observations, given $a\in \mathcal{B}(M,g)$ and $t>0$, the authors in \cite{vdm} define the VDM $V^a_t$ which maps $x\in M$ to the Hilbert space $\ell^2$ by\footnote{Note that the basis $a$ and the volume of $M$ are not taken into consideration in the definition of the VDM in \cite{vdm}. To prove the precompactness theorem, we need to take them into consideration.}:
\begin{equation}\label{Vt}
V^a_t:x \mapsto \Vol(M)\left( e^{-(\lambda_n + \lambda_m)t/2} \langle X_n(x), X_m(x) \rangle \right)_{n,m=1}^\infty,
\end{equation}
where $a=\{X_n\}_{n=1}^\infty$. A direct calculation shows that
\begin{equation}
\|k_{TM}(t,x,y)\|_{HS}^2 = \frac{1}{\Vol(M)^2}\langle V^a_t(x), V^a_t(y)\rangle_{\ell^2}.
\end{equation}
Fix $a\in \mathcal{B}(M,g)$. For all $t>0$, the following Theorem states that the VDM $V^a_t$ is an embedding of the compact Riemannian manifold $M$ into $\ell^2$. The proof of the theorem is given in \cite[Theorem 8.1]{vdm}.
\begin{thm}\label{Vtembedding}
Given a $d$-dim closed Riemannian manifold $(M,g)$ and an orthonormal basis
$a=\{X_k\}_{k=1}^\infty$ of $L^2(T M)$ composed of the eigenvector-fields of the connection
Laplace $\Delta^{\texttt{TM}}$, then for any $t>0$, the VDM $V^a_t$ is a diffeomorphic
embedding of $ M$ into $\ell^2$.
\end{thm}

It is Theorem \ref{Vtembedding} that allows the authors to define the VDD between $x,y\in  M$, denoted as $d_{\text{VDM},t}(x,y)$, in \cite{vdm}:
\begin{equation}
\label{dd}
d_{\text{VDM},t}(x,y) := \|V^a_t(x)-V^a_t(y)\|_{\ell^2},
\end{equation}
which is clearly a distance function over $M$. We mention that by the following expansion
\begin{equation}\label{vddexpansion}
\begin{split}
&~d^2_{\text{VDM},t}(x,y)= \|V^a_t(x)-V^a_t(y)\|^2_{\ell^2}\\
=&~ \Vol(M)^2\sum_{n,m=1}^\infty e^{-(\lambda_n + \lambda_m)t} (\langle X_n(x), X_m(x) \rangle-\langle X_n(y), X_m(y) \rangle)^2\\
=&~\Vol(M)^2\big[\operatorname{Tr}(k_{TM}(t,x,x)k_{TM}(t,x,x)^*)+\\
&~\operatorname{Tr}(k_{TM}(t,y,y)k_{TM}(t,y,y)^*)-2\operatorname{Tr}(k_{TM}(t,x,y)k_{TM}(t,x,y)^*)\big],
\end{split}
\end{equation}
the defined VDD $d_{\text{VDM},t}$ does not depend on the choice of the basis $a$. The following theorem shows that in this asymptotic limit the vector diffusion distance behaves like the geodesic distance. The proof of the theorem is given in \cite[Theorem 8.2]{vdm}.
\begin{thm}\label{thm:geod}
Let $(M,g)$ be a smooth $d$-dim closed Riemannian manifold. Suppose $x,y\in  M$ so that $x=\exp_yv$, where $v\in T_y M$. For any $t>0$, when $\|v\|^2\ll t\ll 1$ we have the following asymptotic expansion of the VDD:
\begin{align}
d^2_{\textup{VDM},t}(x,y)= d\Vol(M)^2(4\pi)^{-d}\frac{\|v\|^2}{t^{d+1}}+O(t^{-d}\|v\|^2).
\end{align}
\end{thm}

\section{Vector Spectral Distances}\label{section:vsd}

In this section and the next, we show that based on the VDM $V^a_t$, we can define a family of VSD $d_t$, $t>0$, on the space of the isometry classes in $\MM_{d,k,D}$ so that for any $t>0$ the space of the isometry classes in $\MM_{d,k,D}$ is $d_t$-precompact.

Denote the Laplace-Beltrami operator over $(M,g)$ by $\Delta_ M$ and its eigenvalues and eigenfunctions by $-\mu_k$ and $\phi_k$, where $k\in\{0\}\cup\NN$, that is, $\Delta_ M \phi_k=-\mu_k\phi_k$, so that $\mu_0=0<\mu_1\leq\mu_2\ldots$. Define the following partition functions
\begin{equation}
Z_{TM}(t):=\sum_{j=1}^\infty e^{-\lambda_jt}
\end{equation}
and 
\begin{equation}
Z_M(t):=\sum_{j=0}^\infty e^{-\mu_jt},
\end{equation}
which are related by the following generalized Kato's type inequality \cite[p. 135]{berard2}:
\begin{equation}\label{kato}
Z_{TM}(t)\leq dZ_M(t)\mbox{ for all }t> 0.
\end{equation}
Then recall the following result:
\begin{thm}\cite[p.108 C.26]{berard2}\label{thm1}
Let $(M,g)$ be an $d$-dimensional closed Riemannian manifold. Define
\begin{equation}
r_{min}(M)=\inf\{\Ric(v,v): \|v\|=1\}
\end{equation}
and the diameter of $M$ by $D(M)$. If $(M,g)$ satisfies $r_{min}(M)D(M)^2\geq (d-1)\epsilon\alpha^2$ for $\epsilon\in\{-1,0,1\}$ and $\alpha>0$, then
\begin{equation}\label{compareheatkernel}
\Vol(M)k_M(t,x,x)\leq \Vol(S^d(R))k_{S^d}(t,y,y)=Z_{S^d(R)}(t)=Z_{S^d(1)}(t/R^2),
\end{equation}
where $y\in S^d(R)$, $S^d(R)$ is the standard $d$-dim sphere in $\RR^{d+1}$ with center $0$ and radius $R$, $R=D(M)/a(d,\epsilon,\alpha)$ and
\begin{equation}
a(d,\epsilon,\alpha)=\left\{
\begin{array}{lcl}
\alpha \omega^{1/d}_d\left(2\int^{\alpha/2}_0 \cos^{d-1}(t)\ud t\right)^{-1/d}& \mbox{ if }&\epsilon=1\\
(1+d\omega_d)^{1/d}-1 & \mbox{ if }&\epsilon=0\\
\alpha c(\alpha)&\mbox{ if }&\epsilon=-1,
\end{array}
\right. 
\end{equation}
where $\omega_d=\Vol(S^d)/\Vol(S^{d-1})$ and $c(\alpha)$ is the unique positive root $z>0$ of the equation
\begin{equation}
z\int^\alpha_0(\cosh(t)+z\sinh(t))^{d-1}\ud t=\omega_d.
\end{equation}
\end{thm}

With inequalities (\ref{kato}) and (\ref{compareheatkernel}), we prove the following lemmas, which is essential in showing the pre-compactness result. We omit the dependence on $M$ to simplify the statement of the lemmas and the proof.
\begin{lem}\label{evlemma}
With the above notations, there exist positive constants $A(d,k,D),B(d,k,D)$ and $E(d,k,D)$ that depend only on $d,k$ and $D$ such that for any $(M,g)\in \MM_{d,k,D}$:\\
\quad$(\textup{a})$ $\lambda_j\geq A(d,k,D)j^{2/d}$;\\
\quad$(\textup{b})$ $N(\lambda):=\#\{j~|~j\geq 0,~\lambda_j\leq\lambda\}\leq ed+B(d,k,D)\lambda^{d/2}$;\\
\quad$(\textup{c})$ for all $x\in M$ and $\alpha\geq 0$, we have
\begin{equation}\label{evlemma3}
\sum_{n,m\geq 1}(\lambda_n+\lambda_m)^\alpha e^{-t(\lambda_n+\lambda_m)}\langle X_n(x),X_m(x)\rangle^2\leq \frac{E(d,k,D)}{\Vol(M)^2}F(\alpha,d)t^{-\alpha-d},
\end{equation}
where
\begin{equation}\label{lemma:1:def:F}
F(\alpha,d):=\int_0^\infty\int_0^\infty[(x+y)^2-2\alpha(x+y)+\alpha(\alpha-1)](x+y)^{\alpha+d-2} e^{-(x+y)}\ud x\ud y.
\end{equation}
\end{lem}

\begin{proof}
The proofs for (a) and (b) are almost the same as the proofs of Theorem 3 in \cite{berard} except that we apply (\ref{kato}). We provide the proofs here for completion. If $k\geq 0$, $r_{min}D^2\geq 0$ and if $k<0$, $r_{min}D^2\geq (d-1)kD$. Thus we can apply Theorem \ref{thm1} with $\epsilon$ and $\alpha$ depending only on $k$ and $D$. Thus,
\begin{align}
 &Z_{TM}(t)\leq dZ_M(t)=d\int_Mk_M(t,x,x)\ud x\leq d\Vol(M)\sup_{x\in M}k_M(t,x,x)\nonumber\\
 \leq& d\Vol(S^d(R))k_{S^d(R)}(t,y,y)=dZ_{S^d(R)}(t)=dZ_{S^d(1)}(t/R^2),
\end{align}
where $y\in S^d(R)$.
The trace of the heat kernel $Z_{TM}$ is thus uniformly bounded on the set $\MM_{d,k,D}$. Note that for $S^d(1)$, the $j$-th eigenvalue is $-j(j+d-1)$ with multiplicity $\left(\hspace{-6pt}\begin{array}{c}d+j\\d\end{array}\hspace{-6pt}\right)-\left(\hspace{-6pt}\begin{array}{c}d+j-2\\d\end{array}\hspace{-6pt}\right)$, where $\left(\hspace{-6pt}\begin{array}{c}m\\n\end{array}\hspace{-6pt}\right)$ is the number of ways of choosing $n$ from $m$,
so by a direct calculation, there exists a constant $b(d)$, depending on $d$ only, such that for any $t>0$,
\begin{equation}
Z_{S^d(1)}(t)-1\leq b(d)t^{-d/2}.
\end{equation}
Also note that $j\leq N(\lambda_j)$, $j\in\NN\cup\{0\}$, in general, and $j=N(\lambda_j)$ when all eigenvalues are simple. As a consequence, we have
\begin{eqnarray}
&&\quad j\leq N(\lambda_j)\leq e\sum_{0\leq \lambda_i\leq\lambda_j}e^{-\lambda_i/\lambda_j}\leq eZ_{TM}(1/\lambda_j)\nonumber\\
&&\leq edZ_{S^d(1)}\left(\frac{1}{\lambda_jR^2}\right)\leq ed+edb(d)R^d\lambda_j^{d/2}
\end{eqnarray}
and hence
\begin{equation}
\lambda_j\geq \left(\frac{j-ed}{edb(d)R^d}\right)^{2/d}\geq (edb(d)R^d)^{-2/d}j^{2/d}.
\end{equation}
Since $R=a(d,\epsilon,\alpha)D(M)\leq a(d,\epsilon,\alpha)D$ and $a(d,\epsilon,\alpha)$ only depends on $d,k$ and $D$, this proves (1) and (2).

Next we prove (c). Define the positive measure $\mu_x$, where $x\in M$, on $\RR\times\RR$ by
\begin{equation}
\ud\mu_x=\sum_{n,m\geq 1}\langle X_n(x),X_m(x)\rangle^2\delta_{\lambda_n}\times\delta_{\lambda_m}
\end{equation}
where $\delta_{\lambda_n}$ is the Dirac measure at $\lambda_n\in\RR$. 
The left hand side of (\ref{evlemma3}) becomes:
\begin{eqnarray}
&&\sum_{n,m\geq 1}(\lambda_n+\lambda_m)^\alpha e^{-t(\lambda_n+\lambda_m)}\langle X_n(x),X_m(x)\rangle^2\nonumber\\
&=&\int_{\RR}\int_{\RR} (\lambda+\nu)^\alpha e^{-t(\lambda+\nu)} \ud\mu_x(\lambda,\nu)\label{evlemmaineq1},
\end{eqnarray}
where the equality holds since when $t>0$, $(\lambda+\nu)e^{-t(\lambda+\nu)}$ is smooth and decays fast enough; that is, $(\lambda+\nu)e^{-t(\lambda+\nu)}\chi_{[0,\infty)\times [0,\infty)}$ is a Schwartz function defined on $\RR^2$. To study (\ref{evlemmaineq1}),
consider the following $L^1_{loc}(\RR^2)$ function:
\begin{equation}
\psi(\lambda,\nu):=\sum_{n:~0\leq\lambda_n\leq\lambda,~m:~0\leq\lambda_m\leq\nu}\langle X_n(x),X_m(x) \rangle^2,
\end{equation}
which is supported on $[0,\infty)\times [0,\infty)$. Note that by the definition of the derivative in the sense of distribution, for any Schwartz function $\phi$ we have
\begin{equation}
\int \phi(\lambda,\nu)\sum_{n\geq 1,~m\geq 1}\langle X_n(x),X_m(x)\rangle^2\delta_{\lambda_n}\times \delta_{\lambda_m}=\int \partial_\lambda\partial_\nu\phi(\lambda,\nu)\psi(\lambda,\nu)\ud \lambda\ud \nu.
\end{equation}
Thus, (\ref{evlemmaineq1}) becomes
\begin{align}
&\int_{\RR}\int_{\RR} (\lambda+\nu)^\alpha e^{-t(\lambda+\nu)} \ud \mu_x(\lambda,\nu)\nonumber\\
=\,&\int_0^\infty   \int_0^\infty \partial_\lambda \partial_\nu\left[(\lambda+\nu)^\alpha e^{-t(\lambda+\nu)}\right]\psi(\lambda,\nu)\ud \lambda \ud\nu\nonumber\\
=\,&\int_0^\infty\int_0^\infty\left[t^2(\lambda+\nu)^2-2\alpha t(\lambda+\nu)+\alpha(\alpha-1)\right]\times\nonumber\\
&\qquad\qquad(\lambda+\nu)^{\alpha-2}e^{-t(\lambda+\nu)}\psi(\lambda,\nu)\ud\lambda\ud\nu.
\end{align}
To finish the proof, we bound $\psi(\lambda,\nu)$ when $\lambda,\nu>0$. By (\ref{compareheatkernel}) and the fact that $\|k_{TM}(t,x,x)\|_{HS}\leq k_M(t,x,x)$ \cite[p137]{berard} for all $t>0$ and $x\in M$, we have
\begin{align}
\psi(\lambda,\nu)=\,&\sum_{n:~0\leq\lambda_n\leq\lambda,~m:~0\leq\lambda_m\leq\nu}\langle X_n(x),X_m(x)\rangle^2\nonumber\\
\leq\,& e\sum_{n:~0\leq\lambda_n\leq\lambda,~m:~0\leq\lambda_m\leq\nu}e^{-(\lambda_n+\lambda_m)/(\lambda+\nu)}\langle X_n(x),X_m(x)\rangle^2\nonumber\\
\leq\,&e\sum_{n=1,\ldots,\infty,~m=1,\ldots,\infty}e^{-(\lambda_n+\lambda_m)/(\lambda+\nu)}\langle X_n(x),X_m(x)\rangle^2 \nonumber\\
 =\,&e \Big\|k_{TM}\Big(\frac{1}{\lambda+\nu},x,x\Big)\Big\|_{HS}^2\leq  ek_M\Big(\frac{1}{\lambda+\nu},x,x\Big)^2\nonumber\\
\leq\,&\frac{e}{\Vol(M)^2}Z_{S^n}\Big(\frac{1}{(\lambda+\nu)R^2}\Big)\leq\frac{C(d,k,D)}{\Vol(M)^2}(\lambda+\nu)^{d},
\end{align}
where $R$ is defined in Theorem \ref{thm1} and $C(d,k,D)$ is an universal constant depending on $d,k$ and $D$. Thus we conclude that the left hand side of (\ref{evlemma3}) is bounded by
\begin{align}
&\sum_{n,m\geq 1}(\lambda_n+\lambda_m)^\alpha e^{-t(\lambda_n+\lambda_m)}\langle X_n(x),X_m(x)\rangle^2\nonumber\\
\leq\,& \frac{C(d,k,D)^2}{\Vol(M)^2}\int_0^\infty\int_0^\infty\left[t^2(\lambda+\nu)^2-2\alpha t(\lambda+\nu)+\alpha(\alpha-1)\right]\times\nonumber\\
&\qquad\qquad\qquad\qquad\qquad(\lambda+\nu)^{\alpha-2}e^{-t(\lambda+\nu)}(\lambda+\nu)^{d}\ud\lambda\ud\nu\nonumber\\
\leq\,& \frac{E(d,k,D)}{\Vol(M)^2}F(\alpha,d)t^{-\alpha-d},
\end{align}
where $E(d,k,D)$ is an universal positive constant that depends only on $d,k$ and $D$, and $F$ is defined in (\ref{lemma:1:def:F}).
\end{proof}

Recall that the VDM $V^a_t$ depends on the choice of an orthonormal basis $a$ of eigen-vector fields. Given a finite dimensional Euclidean space $E$, we can define the distance between $R_1,R_2\in O(\mbox{dim}E)$ by
\begin{equation}\label{Definition:DistanceBetweenOd}
d_E(R_1,R_2)=\|R_1^{-1}R_2-I\|_{HS}.
\end{equation}
It is clear that $d_E(R_1,R_2)\leq 2\sqrt{\mbox{dim}E}$. As we discussed above, $\mathcal{B}(M,g)$ is a compact set with respect to the product topology. This topology can be described by the distance $d_{\mathcal{B}(M,g)}$ between $a,b\in\mathcal{B}(M,g)$ defined as
\begin{equation}\label{Definition:dBab}
d_{\mathcal{B}(M,g)}(a,b)^2:=\sum_{i=1}^\infty \lambda_i^{-N}d_{E_i}(a|_{E_i},b|_{E_i})^2,
\end{equation}
where $a|_{E_i}\in \mathcal{B}(E_i)$ means a basis of the eigenspace associated with the $i$-th eigenvalue and $d_{E_i}(a|_{E_i},b|_{E_i})$ is defined by (\ref{Definition:DistanceBetweenOd}). The series on the right hand side of (\ref{Definition:dBab}) converges when $N>d/2$ due to Lemma \ref{evlemma}(a). The following Lemma is a generalization of Theorem 8 in \cite{berard} to the tangent bundle we have interest. 

\begin{lem}\label{continuitylemma}
Let $(M,g)$ be a smooth d-dim closed Riemannian manifold. The map $V:\RR_+\times \mathcal{B}(M,g)\times M\rightarrow \ell^2$ defined by $V(t,a,x):=V^a_t(x)$ is continuous and satisfies:
\begin{eqnarray}
&&\|V^a_t(x)-V^b_s(y)\|_{\ell^2}^2 \leq \Vol(M)^2\Big\{\|k_{TM}(t,x,x)\|_{HS}+\|k_{TM}(s,y,y)\|_{HS}\nonumber\\
&&\qquad-2\left\|k_{TM}\left(\frac{t+s}{2},x,y\right)\right\|_{HS}
+2 d_{\mathcal{B}(M,g)}(a,b)k^{(N)}_{TM}(t,x,x)^{1/2}k^{(N)}_{TM}(s,y,y)^{1/2}\Big\}\label{vtdiff}
\end{eqnarray}
where $t>0$, $a,b\in \mathcal{B}(M,g)$, $N>d/2$, $x,y\in M$, and
\begin{equation}\label{defkN}
k^{(N)}_{TM}(t,x,x)=\sum_{n,m=1}^\infty (\lambda_n^{N/2}+\lambda_m^{N/2})e^{-t(\lambda_n+\lambda_m)} \langle X^a_{n}(x),X^a_{m}(x)\rangle^2.
\end{equation}
\end{lem}

\begin{proof}
We denote the basis $a\in \mathcal{B}(M,g)$ by $\{X^a_n\}_{n=1}^\infty$. First we have
\begin{align*}
\begin{split}
&\|V^a_t(x)-V^b_s(y)\|_{\ell^2}^2\\
=&\Vol(M)^2\sum_{n,m=1}^\infty\left(e^{-t(\lambda_n+\lambda_m)/2}\langle X^a_n(x),X^a_m(x)\rangle-e^{-s(\lambda_n+\lambda_m)/2}\langle X^b_n(y),X^b_m(y)\rangle\right)^2\\
=&\Vol(M)^2\Big\{\|k_{TM}(t,x,x)\|^2_{HS}+\|k_{TM}(s,y,y)\|^2_{HS}\\
&~ -2\sum_{n,m=1}^\infty e^{-(t+s)(\lambda_n+\lambda_m)/2}\langle X^a_n(x),X^a_m(x)\rangle\langle X^b_n(y),X^b_m(y)\rangle\Big\}\\
=&\Vol(M)^2\Big\{\|k_{TM}(t,x,x)\|^2_{HS}+\|k_{TM}(s,y,y)\|^2_{HS}-\left\|k_{TM}\left({\frac{t+s}{2}},x,y\right)\right\|^2_{HS}\\
&~ -2\sum_{n,m=1}^\infty e^{-(t+s)(\lambda_n+\lambda_m)/2}\langle X^a_n(x),X^a_m(x)\rangle\Big(\langle X^b_n(y),X^b_m(y)\rangle-\langle X^a_n(y),X^a_m(y)\Big)\Big\},
\end{split}
\end{align*}
where we denote the last summation on the right hand side as $A$. Denote the eigen-vector fields inside the eigenspace $E_n$ by $X_{n(j)}^b(y)$, where $j=1,...,m(\nu_n)$, when the basis of $L^2(TM)$ is chosen to be $b\in\mathcal{B}(M,g)$. By definition, we have the following relationship:
\begin{equation}
X^b_{n(j)}(y)=\sum_{k=1}^{m(\nu_n)}\alpha_{j,k}(b,a)X^a_{n(k)}(y)
\end{equation}
where the matrix $[\alpha_{j,k}(b,a)]_{k,j=1}^{m(\nu_j)}\in O(m(\nu_n))$. Here the orthogonal matrix $[\alpha_{j,k}(b,a)]_{k,j=1}^{m(\nu_j)}$ is the relationship between two bases, $a|_{E_j}$ and $b|_{E_j}$, associated with the eigenspace of the $j$-th eigenvalue. Thus we can rewrite $A$ as
\begin{align}
A&=\sum_{n,m=1}^\infty e^{-(t+s)(\nu_n+\nu_m)/2}\sum_{k=1}^{m(\nu_n)}\sum_{l=1}^{m(\nu_m)}\langle X^a_{n(k)}(x),X^a_{m(l)}(x)\rangle\times\nonumber\\
&\qquad\left(\langle X^a_{n(k)}(y),X^a_{m(l)}(y)\rangle-\sum_{i=1}^{m(\nu_n)}\sum_{j=1}^{m(\nu_m)}\alpha_{k,i}(b,a)\alpha_{l,l}(b,a)\langle X^a_{n(i)}(y),X^a_{n(j)}(y)\rangle\right)\nonumber\\
&=\sum_{n,m=1}^\infty e^{-(t+s)(\nu_n+\nu_m)/2}\sum_{i,k=1}^{m(\nu_n)}\sum_{j,l=1}^{m(\nu_m)}\langle X^a_{n(k)}(x),X^a_{m(l)}(x)\rangle\times\nonumber\\
&\qquad\langle X^a_{n(i)}(y),X^a_{m(j)}(y)\rangle\Big(\delta_{k,i}\delta_{l,j}-\alpha_{k,i}(b,a)\alpha_{l,j}(b,a)\Big),
\end{align}
where $[\delta_{k,l}]_{k,l=1}^{m(\nu_j)}$ is an $m(\nu_j)\times m(\nu_j)$ identity matrix.
Note that
\begin{equation}
\left(\delta_{k,i}\delta_{l,j}-\alpha_{k,i}(b,a)\alpha_{l,j}(b,a)\right)=\delta_{k,i}(\delta_{l,j}-\alpha_{l,j}(b,a))+\alpha_{l,j}(b,a)(\delta_{k,i}-\alpha_{k,i}(b,a))
\end{equation}
which is bounded by $d_{E_n}(a|_{E_n},b|_{E_n})+d_{E_m}(a|_{E_m},b|_{E_m})$. Hence, by the Cauchy-Schwartz inequality, $A$ is bounded by
\begin{align*}
\begin{split}
|A|&\leq\sum_{n,m=1}^\infty e^{-(t+s)(\nu_n+\nu_m)/2}\left(\sum_{k=1}^{m(\nu_n)}\sum_{l=1}^{m(\nu_m)}\langle X^a_{n(k)}(x),X^a_{m(l)}(x)\rangle^2\right)^{1/2}\times\\
&\qquad\left(\sum_{i=1}^{m(\nu_n)}\sum_{j=1}^{m(\nu_m)}\langle X^a_{n(i)}(y),X^a_{m(j)}(y)\rangle^2\right)^{1/2}\Big(d_{E_n}(a|_{E_n},b|_{E_n})+d_{E_m}(a|_{E_m},b|_{E_m})\Big)\\
&\leq 2 d_{\mathcal{B}(M,g)}(a,b)\sum_{n,m=1}^\infty (\nu_n^{N/2}+\nu_m^{N/2})e^{-(t+s)(\nu_n+\nu_m)/2}\times\\
&\qquad\left(\sum_{k=1}^{m(\nu_n)}\sum_{l=1}^{m(\nu_m)}\langle X^a_{n(k)}(x),X^a_{m(l)}(x)\rangle^2\right)^{1/2}\left(\sum_{i=1}^{m(\nu_n)}\sum_{j=1}^{m(\nu_m)}\langle X^a_{n(i)}(y),X^a_{m(j)}(y)\rangle^2\right)^{1/2}\\
&\leq 2 d_{\mathcal{B}(M,g)}(a,b) \left(\sum_{n,m=1}^\infty (\lambda_n^{N/2}+\lambda_m^{N/2})e^{-t(\lambda_n+\lambda_m)}\langle X^a_{n}(x),X^a_{m}(x)\rangle^2\right)^{1/2}\times\\
&\qquad \left(\sum_{n,m=1}^\infty (\lambda_n^{N/2}+\lambda_m^{N/2})e^{-s(\lambda_n+\lambda_m)} \langle X^a_{n}(y),X^a_{m}(y)\rangle^2\right)^{1/2}\\
&=2 d_{\mathcal{B}(M,g)}(a,b) k_{TM}^{(N)}(t,x,x)^{1/2}k_{TM}^{(N)}(s,y,y)^{1/2},
\end{split}
\end{align*}
where $k_{TM}^{(N)}(t,x,x)$ is defined in (\ref{defkN}). Due to Lemma \ref{evlemma}(c), $k_{TM}^{(N)}(t,x,x)$ is bounded, and thus the proof is finished.
\end{proof}

With the above preparation, we are ready to introduce the VSD. Recall the following definitions. Suppose $(X,\delta)$ is s metric space, where $\delta$ is the metric, and $A,B\subset X$. The distance between $A$ and $B$ is defined as:
\begin{equation}
h(A,B):=\inf\{\delta(a,b):a\in A,b\in B\}.
\end{equation}
Denote the generalized ball of radius $\epsilon>0$ around $A$:
\begin{equation}
\mathcal{N}(A,\epsilon):=\{x\in X:h(x,A)<\epsilon\}.
\end{equation}
Given two subsets $A,B\subset X$, we can define the Hausdorff distance, denoted as $\HD$, associated with $\delta$ by
\begin{equation}
\HD(A,B)=\inf\{\epsilon:A\subset \mathcal{N}(B,\epsilon),\, B\subset \mathcal{N}(A,\epsilon)\},
\end{equation}
or equivalently
\begin{equation}\label{hausdorffdef}
\HD(A,B)=\max\Big\{\sup_{x\in A}\inf_{y\in B}\delta(x,y),\, \sup_{y\in B}\inf_{x\in A}\delta(x,y)\Big\}.
\end{equation}

In the following, we focus on the metric space $(\ell^2,\|\cdot\|_{\ell^2})$. Let $\HD$ denote the Hausdorff distance between compact subsets of $\ell^2$ associated with $\|\cdot\|_{\ell^2}$. Given two Riemannian manifolds $M$ and $M'$ and $t>0$, we define a family of functions $d_t:\MM_{d,k,D}\times \MM_{d,k,D}\mapsto \RR$ by
\begin{align}\label{VtM}
\begin{split}
d_t(M,M'):=\max&\left\{\sup_{a\in\mathcal{B}(M,g)}\inf_{a'\in\mathcal{B}(M',g')}\HD(V^a_t(M),V^{a'}_t(M')),\right.\\
&\left.\quad\sup_{a'\in\mathcal{B}(M',g')}\inf_{a\in\mathcal{B}(M,g)}\HD(V^{a'}_t(M'),V^a_t(M'))\right\}.
\end{split}
\end{align}

We have to justify that $d_t$, $t>0$, is a distance function defined on the set consisting of the isometry classes in $\MM_{d,k,D}$, and then call it VSD. The following lemmas are needed for the justification.

\begin{lem}\label{newL2basis}
Let $(M,g)$ be a smooth d-dim closed Riemannian manifold and $\{X_n\}_{n\in\NN}$ be an orthonormal basis of $L^2(TM)$ constituted of the eigen-vector fields of the connection Laplacian. Then 
\begin{equation}
\Span\left\{\langle X_n,X_m\rangle:~n,m\in\NN,\, n\neq m\right\}=L^2(M)\backslash \{\RR\}\,
\end{equation}
that is, all $L^2$ functions over $M$ without the constant functions. 
\end{lem}
\begin{proof}
Fix $f\in L^2(M)$. If $\int_M \langle X_n,X_m\rangle(x) f(x)\ud x=\int_M \langle fX_n,X_m\rangle(x)\ud x=0$ for all $n\neq m$, we show that $f$ is constant. Since $\{X_n\}_{n\in\NN}$ is an orthonormal basis of $L^2(TM)$, we can rewrite $fX_n=\sum_{k\in\NN}\alpha_{n,k}X_k$, and hence we have
\begin{equation}
0=\int_M \langle fX_n,X_m\rangle \ud x=\sum_{k\in\NN}\alpha_{n,k}\int_M \langle X_k,X_m\rangle \ud x=\alpha_{n,m} 
\end{equation}
when $m\neq n$. As a result, we know $fX_n=\alpha_{n,n}X_n$, which implies that $f$ is constant. 
\end{proof}

\begin{lem}\label{expandatangentplane}
Let $(M,g)$ be a smooth d-dim closed Riemannian manifold and $\{X_n\}_{n\in\NN}$ is an orthonormal basis of $L^2(TM)$ constituted of the eigen-vector fields of the connection Laplacian. For any $x_0\in M$, there exist $d$ pairs of $\{(n_i,m_i)\}_{i=1}^d$, where $n_i,m_i\in\NN$, so that the gradient vectors $\nabla\langle X_{n_i}, X_{m_i}\rangle(x_0)$ span $T_{x_0} M$.
\end{lem}
\begin{proof}
If not, there is a vector $v\in T_{x_0}M$ perpendicular to $\Span\{\nabla\langle X_{n}, X_{m}\rangle(x_0)\}_{n,m\in\NN}$. It is well known that any function $u\in C^\infty(M)$ can be expanded by the eigenfunctions of the Laplace Beltrami operator, $\{\phi_i\}_{i=0}^\infty$, that is, $u=\sum_{j=0}^{K-1}\phi_j+\sum_{i=K}^\infty u_i\phi_i$, where $K$ is the number of connected components, $\Delta_M\phi_j=0$ for all $j=0,1,\ldots,K-1$, and $u_i=\int_M u\phi_i\ud x$ for all $i=K,K+1,\ldots$. It follows that $\nabla u(x_0)=\sum_{i=K}^\infty u_i\nabla \phi_i(x_0)$. Since any vector $v\in T_{x_0}M$ can be written as $\nabla u(x_0)$ for some smooth function $u$, we know
\begin{equation}
v=\sum_{i=K}^\infty v_i\nabla \phi_i(x_0)
\end{equation}
for some constants $v_i$. On the other hand, by Lemma \ref{newL2basis}, $\phi_i$, $i=K,K+1,\ldots$, can be expanded by $\{\langle X_{n}, X_{m}\rangle\}_{n,m\in\NN, n\neq m}$. Thus, there exist constants $w_{i,n,m}$ so that
\begin{align}
v&=\sum_{i=K}^\infty v_i\sum_{n,m=1}^\infty w_{i,n,m}\nabla \langle X_{n}, X_{m}\rangle(x_0)\nonumber\\
&=\sum_{n,m=1}^\infty \left(\sum_{i=K}^\infty v_iw_{i,n,m}\right)\nabla \langle X_{n}, X_{m}\rangle(x_0),
\end{align}
which is absurd, and we finish the proof.
\end{proof}

\begin{thm}\label{mainthm}
For any fixed $t>0$, $d_t$ is a distance between isometry classes of Riemannian manifolds in $\MM_{d,k,D}$. In particular, two Riemannian manifolds $(M,g),(M',g')\in \MM_{d,k,D}$ satisfy $d_{t}(M, M')=0$ if and only if $(M,g)$ and $(M',g')$ are isometric.
\end{thm}
\begin{proof}
By definition, it is clear that $d_t(M,M')\geq 0$, $d_t(M,M')=d_t(M',M)$ and the triangular inequality holds. To finish the proof that $d_t$ is a distance we need to show that $d_t(M,M')=0$ if and only if $M$ is isometric to $M'$. If $M$ and $M'$ are isometric, then it is trivial to see that $d_t(M,M')=0$. Now we consider the opposite direction. Fix $t>0$. If $d_{t}(M, M')=0$, we claim that $M$ is isometric to $M'$. By the definition of $d_t$, we know 
\begin{equation}
\inf_{a\in\mathcal{B}(M,g)}\HD(V^a_t(M),V^{a'}_t(M'))=0
\end{equation}
for a given $a'\in\mathcal{B}(M',g')$. Thus there exists a sequence $a_n\in\mathcal{B}(M,g)$, $n=1,2,\ldots$, so that 
\begin{equation}
\lim_{n\to\infty} \HD(V^{a_n}_t(M),V^{a'}_t(M'))=0.
\end{equation}
By the compactness of $\mathcal{B}(M,g)$, a subsequence $\{a_{n_j}\}_{j=1}^\infty$ converges to $a_0\in\mathcal{B}(M,g)$, that is, 
\begin{equation}
\lim_{j\to\infty}d_{\mathcal{B}(M,g)}(a_{n_j},a_0)=0,
\end{equation}
and it follows that $\HD(V^{a_0}_t(M),V^{a'}_t(M'))=0$.

Let $a_0=\{X_n\}_{n\in\NN}$ and $a'=\{X_n'\}_{n\in\NN}$. From the definition of Hausdorff distance and the compactness of $V^{a_0}_t(M)$ and $V^{a'}_t(M')$, we have
\begin{eqnarray}
& &\qquad\mbox{for all }x\in M,\mbox{ there exists }y'_t\in M'\mbox{ s.t. for all }n,m\geq 1 \label{ftdef}\\
& &\Vol(M)e^{-(\lambda_n + \lambda_m)t/2} \langle X_n(x), X_m(x) \rangle=\Vol(M')e^{-(\lambda'_n + \lambda'_m)t/2} \langle X'_n(y'_t), X'_m(y'_t) \rangle\nonumber
\end{eqnarray}
\begin{eqnarray}
& &\qquad\mbox{for all }y'\in M',\mbox{ there exists }x_t\in M\mbox{ s.t. for all }n,m\geq 1 \nonumber\\
& &\Vol(M)e^{-(\lambda_n + \lambda_m)t/2} \langle X_n(x_t), X_m(x_t) \rangle=\Vol(M')e^{-(\lambda'_n + \lambda'_m)t/2} \langle X'_n(y'), X'_m(y') \rangle\nonumber
\end{eqnarray}
Because $\{\langle X_n(x), X_m(x) \rangle\}_{n,m=1}^\infty$ (resp. $\{\langle X_n(y'), X_m(y') \rangle\}_{n,m=1}^\infty$) separate the points\footnote{A set $\{\langle X_n(x), X_m(x) \rangle\}_{n,m=1}^\infty$ separates points provided that for $x\neq y$ there exists $f\in\{\langle X_n(x), X_m(x) \rangle\}_{n,m=1}^\infty$ such that $f(x)\neq f(y)$.} in the manifold by Theorem \ref{Vtembedding}, the point $y'_t$ (resp. $x_t$) is uniquely defined and hence the corresponding map $f_t:x\rightarrow y'_t$ (resp. $h_t:y'\rightarrow x_t$) is well-defined, and it is clear that $f_t$ and $h_t$ are inverse to each other. It is also clear that $f_t$ and $h_t$ are continuous. Indeed, by Lemma \ref{thm:geod}, the geodesic distances between $x,\bar{x}$ and $y'_t=f_t(x),\bar{y}'_t=f_t(\bar{x})$ are related by:
\begin{align}
d_g(x,\bar{x})=&\frac{(4\pi)^{d/2}}{d^{1/2}\Vol(M)}t^{\frac{d+1}{2}}d_{\text{VDM},t}(x,\bar{x})(1+O(t))\nonumber\\
=&\frac{(4\pi)^{d/2}}{d^{1/2}\Vol(M')}t^{\frac{d+1}{2}}d_{\text{VDM},t}(y'_t,\bar{y}'_t)(1+O(t))=d_{g'}(y'_t,\bar{y}'_t)(1+O(t)),
\end{align}
which implies the continuity of $f_t$. Similarly we get the continuity of $h_t$. Then, we show that $f_t$ and $h_t$ are $C^\infty$ diffeomorphism. Define a map $F: M\times M'\rightarrow\RR^d$ by
\begin{equation}
F(x,y')=(\langle X_{n_i}(x), X_{m_i}(x)\rangle-c_{n_i,m_i}(t)\langle X'_{n_i}(y'), X'_{m_i}(y')\rangle)_{i=1}^d,
\end{equation}
where
\begin{equation}
c_{n_i,m_i}(t)=\frac{\Vol(M')}{\Vol(M)}e^{(\lambda_{n_i}+\lambda_{m_i}-\lambda'_{n_i}-\lambda'_{m_i})t/2}.
\end{equation}
Note that $F(h_t(y'),y')=0$. Let $y'_0=f_t(x_0)$. From Lemma \ref{expandatangentplane}, it follows that the partial differentiation of $F$ with related to the first variable at $(x_0,y_0)$ is an isomorphism and hence $h_t$ is locally smooth at $y_0'$ by the implicit function theorem. It follows that $h_t$ is smooth. The same proof shows that $f_t$ is smooth, too.

Next, we show that $\Vol(M)=\Vol(M')$. Denote the induced volume form $(f_t)_*\ud V_M$ by $a_t\ud V_{M'}$, where $a_t$ is smooth. Integrating the relation (\ref{ftdef}), we obtain for $n\neq m$, $n,m\in\NN$:
\begin{align}
0=\,&\Vol(M)e^{-(\lambda_n + \lambda_m)t/2} \int_M\langle X_n(x), X_m(x) \rangle\ud V_M(x)\nonumber\\
=\,&\Vol(M')e^{-(\lambda'_n + \lambda'_m)t/2} \int_M\langle X'_n(f_t(x)), X'_m(f_t(x)) \rangle\ud V_M(x)\nonumber\\
=\,&\Vol(M')e^{-(\lambda'_n + \lambda'_m)t/2} \int_{M'}\langle X'_n(y), X'_m(y) \rangle a_t(y)\ud V_{M'}(y),
\end{align}
and similarly for $n=m$, $n\in\NN$:
\begin{align}
1=\,&\Vol(M)e^{-\lambda_nt} \int_M\langle X_n(x), X_n(x) \rangle\ud V_M(x)\nonumber\\
=\,&\Vol(M')e^{-\lambda'_nt} \int_M\langle X'_n(f_t(x)), X'_n(f_t(x)) \rangle\ud V_M(x)\nonumber\\
=\,&\Vol(M')e^{-\lambda'_nt} \int_{M'}\langle X'_n(y), X'_n(y) \rangle a_t(y)\ud V_{M'}(y).
\end{align}
In other words, we have when $n\neq m$,
\begin{equation}\label{inner1}
\int_{M'}\langle X'_n(y), X'_m(y) \rangle a_t(y)\ud V_{M'}(y)=0,
\end{equation}
and when $n=m$,
\begin{equation}\label{inner2}
\int_{M'}\langle X'_n(y), X'_n(y) \rangle a_t(y)\ud V_{M'}(y)=1.
\end{equation}
We claim that from (\ref{inner1}) and (\ref{inner2}), $a_t=1$. Indeed, since $a_t$ is smooth and $\{X_k\}_{k=1}^\infty$ form an orthonormal basis of $L^2(TM)$, by Lemma \ref{newL2basis} and (\ref{inner1}) we conclude that $a_t(y)$ is constant from (\ref{inner1}). From (\ref{inner2}), we know $a_t(y)=1$. Hence,
\begin{eqnarray}
\Vol(M)=\int_M\ud V_M(x)=\int_{M'}(f_t)_*\ud V_{M'}(y)=\Vol(M').
\end{eqnarray}
Plug $\Vol(M)=\Vol(M')$ into (\ref{ftdef}). Integrating (\ref{ftdef}) gives $e^{-(\lambda_n+\lambda_m)t/2}=e^{-(\lambda'_n+\lambda'_m)t/2}$, which implies $e^{-\lambda_nt}=e^{-\lambda'_nt}$ for all $n\geq 1$ and hence $\lambda_n=\lambda'_n$ for all $n\geq 1$.

So far, (\ref{ftdef}) becomes: for all $n,m=1,2,\ldots$, $\lambda_n=\lambda'_n$ and
\begin{equation}\label{finalequation}
\langle X_n(x),X_m(x)\rangle=\langle X'_n(f_t(x)),X'_m(f_t(x))\rangle.
\end{equation}
We now show that $f_t$ is an isometry. Fix $p\in M$. Note that since $f_t$ is a diffeomorphism, we can find an orthonormal frame $\{E_i\}_{i=1}^d$ around $p$ and $\{E'_i\}_{i=1}^d$ around $f_t(p)$ so that $\ud f_t|_pE_i=a_i(p)E'_i(f_t(p))$, where $a_i(p)>0$. 
Indeed, we can find normal coordinates around $p$ and $f(p)$ so that we have orthonormal bases $\{e_i\}_{i=1}^d$ of $T_pM$ and $\{e'_i\}_{i=1}^d$ of $T_{f_t(p)}M'$. Rewrite the linear transformation ${f_t}_*|_p$ in the matrix form, denoted as $[{f_t}_*|_p]$, with related to these bases. By the singular value decomposition we have $[{f_t}_*|_p]=U\Lambda V^T$, where $V,U\in O(d)$. Clearly, $\{Ve_i\}$ and $\{Ue'_i\}$ are orthonormal bases of $T_pM$ and $T_{f_t(p)}M'$ respectively. The orthonormal frames $E_i$ and $E_i'$ are thus defined so that $E_i(p)=Ve_i$ and $E_i'(f_t(p))=Ue_i'$. Denote $a_i(p)=\Lambda(i,i)$. 
To finish the proof, we have to show that $a_i(p)=1$ for all $i=1,\ldots,d$. 
Choose a vector field $Z\in C^\infty(TM)$ so that
\begin{equation}\label{almostdoneZ}
\left\{\begin{array}{l}
Z(p)=0,~~\nabla Z(p)=0,\\
\nabla^2_{E_1,E_1}Z(p)\neq 0,\\
\nabla^2_{E_1,E_1}Z(p)+\nabla^2_{E_2,E_2}Z(p)=0,\\
\nabla^2_{E_l,E_l}Z(p)=0\mbox{ for all }l=3,\ldots,d\,.
\end{array}\right.
\end{equation}
Since $Z\in C^\infty(TM)$, we have $Z=\sum_{n=1}^\infty \alpha_nX_n$. Also, by the construction, we have $\Delta^{\texttt{TM}} Z(p)=0$.
To show the existence of this vector field, we could consider, for example, $Z=(x_1^2-x_2^2)(\partial/\partial x_1)+(x_1^2-x_2^2)(\partial/\partial x_2)$, where $\{x_i\}_{i=1}^d$ is the normal coordinate around $p$ so that $(\partial/\partial x_i)(p)=E_i(p)$ for $i=1,\ldots,d$. Note that we have $x_i(p)=0$, locally near $p$ we have $\partial/\partial x_j(x) = E_j(x)-\frac{1}{6}\sum_ix_kx_l R_{kilj}(p)E_i(x)+\ldots$, where $R_{kilj}$ is the Riemannian curvature, and $\nabla_{E_i}(\partial/\partial x_j)(p)=0$. 

Denote the Levi-Civita connection of $M'$ as $\nabla'$ and the associated connection Laplacian as ${\Delta'}^{\texttt{TM}}$. Construct $Z'\in C^\infty(TM')$ by $Z'=\sum_{n=1}^\infty \alpha_nX'_n$, where $\{X'_n\}$ is the eigen-vector fields of ${\Delta'}^{\texttt{TM}}$. We claim that
\begin{equation}\label{almostdone0}
\left\{\begin{array}{l}
Z'(f_t(p))=0,~~\nabla' Z'(f_t(p))=0,\\
{\nabla'}^2_{E'_1,E'_1}Z'(f_t(p))\neq 0,\\
{\nabla'}^2_{E'_1,E'_1}Z'(f_t(p))+{\nabla'}^2_{E'_2,E'_2}Z'(f_t(p))=0,\\
{\nabla'}^2_{E'_l,E'_l}Z'(f_t(p))=0\mbox{ for all }l=3,\ldots,d\,.
\end{array}\right.
\end{equation}
Note that this claim implies ${\Delta'}^{\texttt{TM}} Z'(f_t(p))=0$. By (\ref{finalequation}), we know 
\begin{align}
\langle Z'(f_t(p)),Z'(f_t(p))\rangle=&\sum_{k,l=1}^\infty \alpha_k\alpha_l\langle X_k'(f_t(p)),X_l'(f_t(p))\rangle\nonumber\\
=&\sum_{k,l=1}^\infty \alpha_k\alpha_l\langle X_k(p),X_l(p)\rangle=\langle Z(p),Z(p)\rangle=0
\end{align}
and
\begin{align}
&\langle {\Delta'}^{\texttt{TM}}Z'(f_t(p)),{\Delta'}^{\texttt{TM}}Z'(f_t(p))\rangle=\sum_{k,l=1}^\infty \alpha_k\lambda'_k\alpha_l\lambda'_l\langle X_k'(f_t(p)),X_l'(f_t(p))\rangle\nonumber\\
=&\sum_{k,l=1}^\infty \alpha_k\lambda_k\alpha_l\lambda_l\langle X_k(p),X_l(p)\rangle=\langle \Delta^{\texttt{TM}}Z(p),\Delta^{\texttt{TM}}Z(p)\rangle=0,
\end{align}
which implies $Z'(f_t(p))=0$ and ${\Delta'}^{\texttt{TM}}Z'(f_t(p))=0$. Take $v\in T_pM$ and a curve $\gamma:[0,\epsilon)\to M$ so that $\gamma(0)=p$ and $\gamma'(0)=v$. By extending $v$ to $V\in \Gamma(TM)$ so that $V(\gamma(t))=\gamma'(t)$, we have
\begin{align}
&{\frac{\ud^2}{\ud t^2}}|_{t=0}\langle Z,Z\rangle(\gamma(t))=2\frac{\ud}{\ud t}|_{t=0}\langle \nabla_V Z,Z\rangle(\gamma(t))\nonumber\\
=\,&2\big(\langle \nabla_V\nabla_V Z,Z\rangle(p)+\langle \nabla_V Z,\nabla_VZ\rangle(p)\big)=2\langle \nabla_V Z,\nabla_VZ\rangle(p),
\end{align}
where the last equality holds since $Z(p)=0$.
On the other hand, by (\ref{finalequation}) we have
\begin{align}
&\frac{\ud^2}{\ud t^2}|_{t=0}\langle Z,Z\rangle(\gamma(t))\nonumber\\
=\,&\frac{\ud^2}{\ud t^2}|_{t=0}\langle Z',Z'\rangle(f_t\circ\gamma(t))\nonumber\\
=\,&2\frac{\ud}{\ud t}|_{t=0}\langle \nabla'_{{f_t}_*V} Z',Z'\rangle(f_t\circ\gamma(t))\nonumber\\
=\,&2\big(\langle \nabla'_{{f_t}_*V}\nabla'_{{f_t}_*V} Z',Z'\rangle(f_t(p))+\langle \nabla'_{{f_t}_*V} Z',\nabla'_{{f_t}_*V}Z'\rangle(f_t(p))\big)\nonumber\\
=\,&2\langle \nabla'_{{f_t}_*V} Z',\nabla'_{{f_t}_*V}Z'\rangle(f_t(p))\nonumber,
\end{align}
where the last equality comes from the fact that $Z'(f_t(p))=0$. Thus, we have
\begin{equation}
\langle \nabla'_{{f_t}_*V} Z',\nabla'_{{f_t}_*V}Z'\rangle(f_t(p))=\langle \nabla_V Z,\nabla_VZ\rangle(p)=0,
\end{equation}
which implies $\nabla'_{{f_t}_*v} Z'(f_t(p))=0$. Since $f_t$ is diffeomorphic and $v$ is arbitrary, we conclude that $\nabla' Z'(f_t(p))=0$.

Next choose the same curve $\gamma$ and take the fourth order derivative:
\begin{align}
&\frac{\ud^4}{\ud t^4}|_{t=0}\langle Z,Z\rangle(\gamma(t))\nonumber\\
=\,& 2\langle {\nabla}_{V}{\nabla}_{V}{\nabla}_{V}{\nabla}_{V}Z,\,Z\rangle(p)\nonumber\\
&\quad+8\langle {\nabla}_{V}{\nabla}_{V}{\nabla}_{V}Z,\,{\nabla}_{V}Z\rangle(p)\nonumber\\
&\quad+6\langle {\nabla}_{V}{\nabla}_{V}Z,\,{\nabla}_{V}{\nabla}_{V}Z\rangle(p)\nonumber\\
=\,&6\langle {\nabla}_{V}{\nabla}_{V}Z,\,{\nabla}_{V}{\nabla}_{V}Z\rangle(p)\nonumber\\
=\,&6\langle \nabla^2_{V,V}Z+{\nabla}_{\nabla_VV}Z,\,\nabla^2_{V,V}Z+{\nabla}_{\nabla_VV}Z\rangle(p)\nonumber\\
=\,&6\langle \nabla^2_{V,V}Z,\,\nabla^2_{V,V}Z\rangle(p)\nonumber
\end{align}
where the second and the fourth equalities come from the fact that $Z(p)=0$ and $\nabla Z(p)=0$. Similarly, by the fact that $Z'(f_t(p))=0$ and $\nabla' Z'(f_t(p))=0$ we have
\begin{align}
&\frac{\ud^4}{\ud t^4}|_{t=0}\langle Z,Z\rangle(\gamma(t))\nonumber\\
=\,&\frac{\ud^4}{\ud t^4}|_{t=0}\langle Z',Z'\rangle(f_t\circ\gamma(t))\nonumber\\
=\,& 2\langle {\nabla'}_{{f_t}_*V}{\nabla'}_{{f_t}_*V}{\nabla'}_{{f_t}_*V}{\nabla'}_{{f_t}_*V}Z',\,Z'\rangle(f_t(p))\nonumber\\
&\quad+8\langle {\nabla'}_{{f_t}_*V}{\nabla'}_{{f_t}_*V}{\nabla'}_{{f_t}_*V}Z',\,{\nabla'}_{{f_t}_*V}Z'\rangle(f_t(p))\nonumber\\
&\quad+6\langle {\nabla'}_{{f_t}_*V}{\nabla'}_{{f_t}_*V}Z',\,{\nabla'}_{{f_t}_*V}{\nabla'}_{{f_t}_*V}Z'\rangle(f_t(p))\nonumber\\
=\,&6\langle {\nabla'}_{{f_t}_*V}{\nabla'}_{{f_t}_*V}Z',\,{\nabla'}_{{f_t}_*V}{\nabla'}_{{f_t}_*V}Z'\rangle(f_t(p))\nonumber\\
=\,&6\langle {\nabla'}^2_{{f_t}_*V,~{f_t}_*V}Z'+{\nabla'}_{\nabla'_{{f_t}_*V}{f_t}_*V}Z',\,{\nabla'}^2_{{f_t}_*V,{f_t}_*V}Z'+{\nabla'}_{\nabla_{{f_t}_*V}{f_t}_*V}Z'\rangle(f_t(p))\nonumber\\
=\,&6\langle {\nabla'}^2_{{f_t}_*V,{f_t}_*V}Z',\,{\nabla'}^2_{{f_t}_*V,{f_t}_*V}Z'\rangle(f_t(p))\nonumber.
\end{align}
Since $v$ is arbitrary, we have
\begin{equation}
\langle {\nabla}^2_{E_i,E_i}Z,~{\nabla}^2_{E_i,E_i}Z\rangle(p)=\langle {\nabla'}^2_{{f_t}_*E_i,{f_t}_*E_i}Z',~{\nabla'}^2_{{f_t}_*E_i,{f_t}_*E_i}Z'\rangle(f_t(p)).
\end{equation}
Thus, we have shown the claim (\ref{almostdone0}).

Next we claim that $a_1(p)=a_2(p)$. Take another smooth vector field $Y\in C^\infty(TM)$ so that 
\begin{equation}\label{Definition:Proof:Main:Y}
Y=\sum_{n=1}^\infty \gamma_nX_n
\end{equation} 
and $Y(p)\neq 0$ and construct $Y'\in C^\infty(TM)$ so that 
\begin{equation}\label{Definition:Proof:Main:Yp}
Y'=\sum_{n=1}^\infty \gamma_nX'_n. 
\end{equation}
Then by taking the curve $\gamma$ so that $\gamma'(0)=E_i$, where $i=1,\ldots,d$, we have
\begin{align}
&\frac{\ud^2}{\ud t^2}|_{t=0}\langle Z,Y\rangle(\gamma(t))\nonumber\\
=\,&\langle {\nabla}_{E_i}{\nabla}_{E_i}Z,Y\rangle(p)+2\langle {\nabla}_{E_i}Z,{\nabla}_{E_i}Y\rangle(p)+\langle Z,{\nabla}_{E_i}{\nabla}_{E_i}Y\rangle(p)\nonumber\\
=\,&\langle {\nabla}^2_{E_i,E_i}Z+{\nabla}_{\nabla_{E_i}E_i}Z,Y\rangle(p)\nonumber\\
=\,&\langle {\nabla}^2_{E_i,E_i}Z,\,Y\rangle(p)\label{almostdone1}
\end{align}
since $Z(p)=0$ and $\nabla Z(p)=0$.
On the other hand, by the same arguments as those for (\ref{almostdone1}), the construction of $Z'$ and $E_i'$, we have
\begin{align}
&\frac{\ud^2}{\ud t^2}|_{t=0}\langle Z,Y\rangle(\gamma(t))\nonumber\\
=\,&\frac{d^2}{\ud t^2}|_{t=0}\langle Z',Y'\rangle(f\circ \gamma(t))\nonumber\\
=\,& \langle {\nabla'}^2_{{f_t}_*E_i,{f_t}_*E_i}Z',Y'\rangle(f_t(p))\nonumber\\
=\,& a_i(p)^2\langle {\nabla'}^2_{E'_i,E'_i}Z',Y'\rangle(f_t(p))\label{almostdone2}
\end{align}
From (\ref{almostdone1}) and (\ref{almostdone2}), we have
\begin{equation}
\langle {\nabla}^2_{E_i,E_i}Z,Y\rangle(p)=a_i(p)^2\langle {\nabla'}^2_{E'_i,E'_i}Z',Y'\rangle(f_t(p)).
\end{equation}
Hence, by the assumption of $Z$ we have
\begin{align}
0&=\langle \Delta^{\texttt{TM}}Z,Y\rangle(p)=\langle {\nabla}^2_{E_1,E_1}Z+{\nabla}^2_{E_2,E_2}Z,Y\rangle(p)\nonumber\\
&=a_1(p)^2\langle {\nabla'}^2_{E'_1,E'_1}Z',Y'\rangle(f_t(p))+a_2(p)^2\langle {\nabla'}^2_{E'_2,E'_2}Z',Y'\rangle(f_t(p))
\end{align}
Since $Y'(f_t(p))$ is arbitrary, we know 
\begin{equation}\label{almostdone:1and2}
a_1(p)^2{\nabla'}^2_{E'_1,E'_1}Z'(f_t(p))+a_2(p)^2{\nabla'}^2_{E'_2,E'_2}Z'(f_t(p))=0.
\end{equation}
Combining with the fact that
\begin{equation}
{\nabla'}^2_{E'_1,E'_1}Z'(f_t(p))+{\nabla'}^2_{E'_2,E'_2}Z'(f_t(p))=0
\end{equation}
in (\ref{almostdone0}), we know $a_1(p)=a_2(p)$. 

By repeating the arguments from (\ref{almostdoneZ}) to (\ref{almostdone:1and2}), we conclude that 
\begin{equation}
a_1(p)=a_2(p)=\ldots=a_d(p)=:a(p).
\end{equation}
To finish the proof, we choose $W\in C^\infty(TM)$ so that $W=\sum_{l=1}^\infty \beta_lX_l$, $W(p)=0$, $\nabla W(p)=0$ and $\Delta^{\texttt{TM}}W(p)\neq 0$. Construct $W'=\sum_{l=1}^\infty\beta_lX'_l$. By the same argument, we know $W'(f_t(p))=0$ and $\nabla'W'(f_t(p))=0$. Take arbitrary vector fields $Y$ on $M$ and $Y'$ on $M'$ constructed by the same way as those in (\ref{Definition:Proof:Main:Y}) and (\ref{Definition:Proof:Main:Yp}). The same arguments for (\ref{almostdone1}) and (\ref{almostdone2}) hold for $W$; that is, when $\gamma(t)$ is a curve on $M$ so that $\gamma(0)=p$ and $\gamma'(0)=E_i(p)$ we have
\begin{eqnarray}
\frac{\ud^2}{\ud t^2}|_{t=0}\langle W,Y\rangle(\gamma(t))=\langle {\nabla}^2_{E_i,E_i}W,Y\rangle(p)
\end{eqnarray}
and
\begin{align}
&\frac{\ud^2}{\ud t^2}|_{t=0}\langle W,Y\rangle(\gamma(t))=\frac{\ud^2}{\ud t^2}|_{t=0}\langle W',Y'\rangle(f\circ \gamma(t))\nonumber\\
=\,& \langle {\nabla'}^2_{{f_t}_*E_i,{f_t}_*E_i}W',Y'\rangle(f_t(p))=a(p)^2\langle{\nabla'}^2_{E'_i,E'_i}W',Y'\rangle(f_t(p)).
\end{align}
Thus we have
\begin{align}
a(p)^2\langle {\Delta'}^{\texttt{TM}}W',Y'\rangle&=a(p)^2\langle\sum_{i=1}^d{\nabla'}^2_{E_i',E_i'}W',Y'\rangle\nonumber\\
&=\langle\sum_{i=1}^d\nabla^2_{E_i,E_i}W,Y\rangle=\langle \Delta^{\texttt{TM}}W,Y\rangle.
\end{align}
On the other hand, the following equality holds due to the definition of $W'$ and $Y'$:
\begin{align}
\langle {\Delta'}^{\texttt{TM}}W',Y'\rangle&=\langle \sum_{l=1}^\infty\beta_l\lambda'_lX'_l,~\sum_{k=1}^\infty\gamma_kX'_k\rangle= \sum_{l,k=1}^\infty \beta_l\lambda'_l\gamma_k\langle X'_l,X'_k\rangle\nonumber\\
&=\sum_{l,k=1}^\infty \beta_l\lambda_l\gamma_k\langle X_l,X_k\rangle=\langle \Delta^{\texttt{TM}}W,Y\rangle,
\end{align}
which gives us $a(p)=1$ and hence $f_t$ is an isometry. We have thus finished the proof.
\end{proof}

\section{Precompactness of $\MM_{d,k,D}$}\label{section:precompact}
By Theorem \ref{mainthm}, we have a distance, referred to as the VSD, in the space of the isometry classes in $\MM_{d,k,D}$. We finally can state the pre-compactness theorem. We need Lemma \ref{precompactlemma} to finish the proof.
\begin{thm}\label{Theorem:Precompactness}
For any $t>0$, the space of the isometry classes in $\MM_{d,k,D}$ is $d_t$-precompact.
\end{thm}

\begin{proof}
Denote $h^1$ to be the Hilbert space consisting of infinite sequences $\{x_i\}_{i=1,2,\ldots}$ so that $\|\{x_i\}_{i=1,2,\ldots}\|_{h^1}:=\sum_{i=1}^\infty(1+i^{2/d}) |x_i|^2<\infty$. Consider a one-to-one map $\NN\times\NN \to \NN$ which maps $(i,j)$ to $\frac{(i+j-1)(i+j-2)}{2}+j$. Thus we could view the double indexed sequence $\{e^{-(\lambda_i+\lambda_j) t}\langle X_i(x),X_j(x)\rangle^2\}_{i,j=1}^\infty$ as an infinite sequence.
Fix $t>0$. For any $ M\in \MM_{d,k,D}$, $a\in\mathcal{B}(M,g)$ and $x\in M$, we have
\begin{align}
\|V^a_t(x)\|^2_{h^{1}}&= \Vol(M)^2\sum_{i,j\geq 1} \left[1+\Big(\frac{(i+j-1)(i+j-2)}{2}+j\Big)^{2/d}\right]e^{-(\lambda_i+\lambda_j) t}\langle X_i(x),X_j(x)\rangle^2\nonumber\\
&\leq  \Vol(M)^2\sum_{i,j\geq 1} (1+(i+j)^{4/d})e^{-(\lambda_i+\lambda_j) t}\langle X_i(x),X_j(x)\rangle^2\nonumber\\
&\leq  \frac{\Vol(M)^2}{A(d,k,D)^2}\sum_{i,j\geq 1} (A(d,k,D)^2+c(d)(\lambda_i+\lambda_j)^2)e^{-(\lambda_i+\lambda_j) t}\langle X_i(x),X_j(x)\rangle^2\nonumber\\
&\leq  E(d,k,D)F(0,d)t^{-d}+\frac{c(d)E(d,k,D)}{A(d,k,D)^2}F(2,d)t^{-2-d}\nonumber,
\end{align}
where $c(d)$ is a universal constant depending on $d$, the first inequality holds by a simple bound $\frac{(i+j-1)(i+j-2)}{2}+j\leq (i+j)^2$, the second inequality follows from Lemma \ref{evlemma}(a) and the third inequality follows from Lemma \ref{evlemma}(c). Since $c(d)$, $A(d,k,D)$, $E(d,k,D)$, $F(0,d)$ and $F(2,d)$ are universal constants, we know that the set
\begin{equation}
K_0:=\{V^a_t(x)\}_{x\in M, M\in \MM_{d,k,D}, a\in\mathcal{B}(M,g)}\subset h^{1}
\end{equation}
is bounded in $h^{1}\subset \ell^2$, which is hence relative compact inside $\ell^2$ by Rellich's Theorem \cite{Bisgard:2012}. Denote the closure of $K_0$ in $\ell^2$ by $K$. Denote the set of all non-empty closed subsets of $K$ by $\mathcal{F}(K)$, equipped with the Hausdorff distance $\HD$ associated with the canonical metric on $\ell^2$. By Lemma \ref{precompactlemma}, the metric space $(\mathcal{F}(K),\HD)$ is precompact.

By Theorem \ref{Vtembedding}, since $M\in \MM_{d,k,D}$ is compact, $V^a_t(M)$ is compact inside $\ell^2$ for any $a\in\mathcal{B}(M,g)$. Hence $V^a_t(M)\in \mathcal{F}(K)$. Consider a subset $E$ of $\mathcal{F}(K)$ consisting of $V^a_t(M)$, where $M\in \MM_{d,k,D}, a\in\mathcal{B}(M,g)$, that is,
\begin{equation}
E=\{V^a_t(M)\}_{M\in \MM_{d,k,D}, a\in\mathcal{B}(M,g)}\subset \mathcal{F}(K).
\end{equation}
Note that $E$ is precompact with related to the distance $\HD$ since a subset of a compact set is precompact. Given $M\in \MM_{d,k,D}$, define a subset $V_t(M)$ of $E$ consisting of $V_t^a(M)$ for all $a\in \mathcal{B}(M,g)$, that is,
\begin{equation}\label{vtset}
V_t(M):=\{V_t^a(M)\}_{a\in\mathcal{B}(M,g)}\subset E.
\end{equation}
Here we view $V_t^a(M)$ as a point in the set $E$. By Lemma \ref{continuitylemma}, $V_t(M)$ is a closed subset of $E$ with related to the Hausdorff distance $\HD$. Indeed, by the closeness of $\mathcal{B}(M,g)$ and Lemma \ref{continuitylemma} we have
\begin{equation}
\HD(V_t^a(M),V_t^b(M))\leq 2\Vol(M)^2d_{\mathcal{B}(M,g)}(a,b)\sup_{x\in M}|K_{TM}^{(N)}(t,x,x)|,
\end{equation}
which implies the closeness of $V_t(M)$.
Then, consider $\mathcal{F}(E)$ the set of non-empty closed subsets of $E$, equipped with the Hausdorff distance $h_{\HD}$ associated with the distance $\HD$. By Lemma \ref{precompactlemma} again, we conclude that $\mathcal{F}(E)$ is precompact with related to the distance $h_{\HD}$. Finally, the set $\{V_t(M)\}_{M\in \MM_{d,k,D}}$, which is a subset of $\mathcal{F}(E)$, is precompact with related to the Hausdorff distance $h_{\HD}$. 

Notice that by the definition of the Hausdorff distance in (\ref{hausdorffdef}) and Theorem \ref{mainthm}, (\ref{VtM}) is nothing but the Hausdorff distance $h_{\HD}$, that is, 
\begin{equation}
d_t(M, M')=d_{\HD}(V_t(M),V_t(M')),
\end{equation}
and we conclude the proof.

\end{proof}

\section{Manifold Learning}\label{section:manifoldLearning}

Over the last couple of decades, explosive technological advances lead to exponential growth of massive datasets in almost all fields. In addition to the big size, the data is often high dimensional and difficult to quantify its hidden structure. In order to deal with this kind of data, several new techniques and algorithms were proposed in the applied mathematics and statistics literature.  
The geometric object, in particular the differentiable manifold, is commonly considered as a candidate model to capture or approximate the nonlinear structure inside the massive dataset, and several techniques based on spectral graph theory \cite{Fan:1996} are extensively studied in order to capture this structure \cite{belkin_niyogi:2003,Coifman20065}. Here we summarize the algorithm and the relevant theory. 

Suppose we have a dataset $\mathcal{X}=\{x_i\}_{i=1}^n$ independently and uniformly sampled from a random variable. The first challenge to data scientists is finding a suitable metric to quantify the relationship among data points. Consider the special case that $\mathcal{X}$ is located on a $d$-dim closed and smooth manifold $M$ embedded in $\RR^p$. Clearly we may use the canonical distance of $\RR^p$ to study the relationship. Once we have the pairwise relationship among data points, we are able to build up an affinity graph and study the dataset by studying its associated graph Laplacian. Precisely, set the vertices $V=\mathcal{X}$, the edges $E=\{(x_i,x_j)\}$ and the affinity function $\mathsf{w}:E\to \RR^+$ by $\mathsf{w}(x_i,x_j):=e^{-\|x_i-x_j\|^2_{\RR^p}/\epsilon}$, where $\epsilon>0$ is chosen by the user. Here $\mathsf{w}$ could be defined with other suitable kernels, but we present the algorithm with the Gaussian kernel for the simplicity. Second, build up  a $n\times n$ matrix $W$ and a $n\times n$ diagonal matrix $D$ by
\begin{equation}
W_{ij}=\mathsf{w}(x_i,x_j) ,\quad D_{ii}=\sum_{j=1}^nW_{ij},  
\end{equation}
and then construct the graph Laplacian by $L=I_n-D^{-1}W$, where $I_n$ is a $n\times n$ identity matrix. Clearly the constructed graph is undirected and $W$ is symmetric. It has been well studied that as $n\to \infty$ and $\epsilon\to 0$ depending on $n$, the graph Laplacian converges to the Laplace-Beltrami operator in the spectral sense in probability \cite{belkin_niyogi:2007,singer_wu:2013}. If we have further information about the parallel transport among two close points, then we can further build up the {\it connection graph}, which consists of the affinity graph discussed above and a connection function $\mathsf{g}: E\to O(d)$, where 
\begin{equation}
\mathsf{g}(x_i,x_j):=b_i^TP_{x_i,x_j}b_j,
\end{equation}
$P_{x_i,x_j}$ is the parallel transport from $x_j$ to $x_i$, and $b_i:\RR^d\to T_{x_i}M$ can be viewed as a chosen basis of the tangent plane $T_{x_i}M$. In this case, we can establish the GCL in the following way. Recall that GCL stands for graph connection Laplacian. First, build up a $n\times n$ block matrix $\mathsf{S}$ and a $n\times n$ block diagonal matrix $\mathsf{D}$ with $d\times d$ blocks: 
\begin{equation}
\mathsf{S}_{ij}=\mathsf{w}(x_i,x_j)\mathsf{g}(x_i,x_j)\in \RR^{d\times d},\quad \mathsf{D}_{ii}=\sum_{j=1}^n\mathsf{w}(x_i,x_j)I_d\in \RR^{d\times d},
\end{equation}
and hence the GCL by $\mathsf{C}=I_{nd}-\mathsf{D}^{-1}\mathsf{S}$. By the construction of $\mathsf{g}$, we know $\mathsf{S}$ is symmetric.
As is shown in \cite{singer_wu:2013}, under this setup, as $n\to \infty$ and $\epsilon\to 0$ depending on $n$, $\mathsf{C}$ converges to the connection Laplacian in the spectral sense in probability. Furthermore, in \cite[Theorem B1 and Theorem B2]{vdm}, it has been shown that via the techniques including local principal component analysis and rotational alignment, we are able to estimate the parallel transport between two close points directly from the point clouds. In other words, even if the parallel transport information among two close points is missing, we are able to extract it directly from the point clouds. We refer the reader to \cite{Coifman20065,vdm,singer_wu:2013,ElKaroui:2010a,ElKaroui_Wu:2014} for more mathematical details as well as their practical algorithms. Note that all the above discussions are limited to one manifold.

Now one consider a different scenario -- what happens if the dataset $\mathcal{X}$ is more abstract than a subset of vectors? Take a set of $2$-dim shapes embedded in $\RR^3$ for example. In this setup, there has been several metrics proposed in the literature to compare two different shapes, for example \cite{Chazal_Cohen_Guibas_Memoli_Oudot:2009,Sun_Ovsjanikov_Guibas:2009,Memoli:2011} and the abundant literature summarized there. In general, the idea is that we view a shape as a manifold and the discretization of the shape as a set of point cloud sampled from the shape, and study its graph Laplacian or GCL. As is proved in \cite{berard} and this paper, if all the shapes in $\mathcal{X}$ are of fixed dimension, bounded diameter and Ricci curvature, that is, $\mathcal{X}\subset\mathcal{M}_{d,k,D}$, then the SD and VSD among these shapes are indeed metrics distinguishing isometric classes inside $\mathcal{X}$. We can then take these metrics to define the affinity function $\mathsf{w}$ in the affinity graph with $V=\mathcal{X}$ and $E=\{(x_i,x_j)\}$, and study the dataset by the graph Laplacian or other relevant techniques. 

One immediate question we might ask when we take a new metric into account in the analysis -- what is the property of $\mathcal{X}$ under the SD or VSD? The pre-compactness theorem proved in \cite{berard} and this paper provide a partial answer. Indeed, we know that even these shapes might have quite diverse geometric and topological profiles, under the SD or VSD the relationship among these manifolds is fairly restricted. Furthermore, since $\mathcal{M}_{d,k,D}$ is pre-compact under these metrics, we have control over several graph quantities. For example, if the constructed affinity graph is connected, we know the graph diameter is finite, and hence the first non-trivial eigenvalue is also non-trivial \cite[Lemma 1.9]{Fan:1996}. A further study about the pre-compactness property and its application to the data analysis will be reported in the future work. Also, the main difference between the VSD and SD deserves a further study. In particular, since the connection Laplacian is intimately related to the Hodge Laplacian, the topological properties might be reflected in the VSD.

\section{Acknowledgements}
The author acknowledges the support partially by FHWA grant DTFH61-08-C-00028, partially by Award Number FA9550-09-1-0551 from AFOSR and partially by Connaught New Researcher grant 498992. He acknowledges Professor Charlie Fefferman and Professor Amit Singer for their time and inspiring and helpful discussions; in particular Professor Amit Singer, who introduced him the massive data analysis field. He also acknowledges the valuable discussion with Professor G\'{e}rard Besson, Professor Richard Bamler and Dr. Chen-Yun Lin. The author also thanks the anonymous reviewer's constructive and helpful comments.

\bibliographystyle{plain}
\bibliography{connLap}
\end{document}